\documentclass[12pt, reqno,draft]{amsart}

\usepackage{amsmath}
\usepackage{amsthm}
\usepackage{amssymb}
\usepackage{amsrefs}
\usepackage{mathrsfs}
\usepackage[usenames]{color}

 \newtheorem{thm}{}[section]
 \newtheorem{theorem}[thm]{Theorem}
 \newtheorem{corollary}[thm]{Corollary}
 \newtheorem{lemma}[thm]{Lemma}
 \newtheorem{proposition}[thm]{Proposition}

 \theoremstyle{remark}
 \newtheorem{remark}[thm]{Remark}
 
 \newtheorem{example}[thm]{Example}
  \newtheorem{definition}[thm]{Definition}
 
 \numberwithin{equation}{section}
 \allowdisplaybreaks
 
 \newcommand{\dd}{\ensuremath{\mathbf{d}}}
\newcommand{\DD}{\ensuremath{\mathcal{D}}}
 \newcommand{\ww}{\ensuremath{\mathbf{w}}}
 \newcommand{\JJ}{\ensuremath{\mathcal{J}}}
  \newcommand{\FF}{\ensuremath{\mathbb{F}}}
\newcommand{\PX}{\ensuremath{\mathcal{W}}}
\newcommand{\BV}{\ensuremath{\mathrm{BV}}}
\newcommand{\VMO}{\ensuremath{\mathrm{VMO}}}
\newcommand{\SL}{\ensuremath{\mathscr{L}}}
 \newcommand{\SSS}{\ensuremath{\mathcal{S}}}
 \newcommand{\SSB}{\ensuremath{\mathbb{S}}}
  \newcommand{\sss}{\ensuremath{\mathbf{s}}}
 \newcommand{\LL}{\ensuremath{\mathcal{L}}}
 \newcommand{\li}{\ensuremath{\mathbf{l}}}
 
 \newcommand{\Id}{\operatorname{Id}}
 \newcommand{\NN}{\ensuremath{\mathbb{N}}}
 \newcommand{\UU}{\ensuremath{\mathbb{U}}}
\newcommand{\XX}{\ensuremath{\mathbb{X}}}
\newcommand{\YY}{\ensuremath{\mathbb{Y}}}
\newcommand{\yy}{\ensuremath{\mathbf{y}}}
\newcommand{\BB}{\ensuremath{\mathcal{B}}}
\newcommand{\OO}{\operatorname{\mathcal{O}}}
\newcommand{\ee}{\ensuremath{\mathbf{e}}}
\newcommand{\WW}{\ensuremath{\mathbb{W}}}
\newcommand{\VV}{\ensuremath{\mathbb{V}}}
\newcommand{\ZZ}{\ensuremath{\mathbb{Z}}}
\newcommand{\xx}{\ensuremath{\mathbf{x}}}
\newcommand{\zz}{\ensuremath{\mathbf{z}}}
\newcommand{\RR}{\ensuremath{\mathbb{R}}}

\newcommand{\GOW}{\operatorname{\mathrm{O}}}

\newcommand{\supp}{\operatorname{supp}}

\begin{document}

\title[Banach spaces with highly conditional quasi-greedy bases]{Building  highly conditional quasi-greedy bases in
classical Banach spaces}

\author[F. Albiac]{Fernando Albiac}
\address{Mathematics Department\\ 
Universidad P\'ublica de Navarra\\
Campus de Arrosad\'{i}a\\
Pamplona\\ 
31006 Spain}
\email{fernando.albiac@unavarra.es}

\author[J. L. Ansorena]{Jos\'e L. Ansorena}
\address{Department of Mathematics and Computer Sciences\\
Universidad de La Rioja\\ 
Logro\~no\\
26004 Spain}
\email{joseluis.ansorena@unirioja.es}



\begin{abstract} It is known that for a conditional quasi-greedy basis $\BB$ in a Banach space $\XX$,  the associated sequence $(k_{m}[\BB])_{m=1}^{\infty}$ of its conditionality constants verifies the estimate $k_{m}[\BB]=\OO(\log m)$ and that if the reverse inequality $\log m =\OO(k_m[\BB])$ holds then $\XX$ is non-superreflexive. However, in the existing literature one finds very few instances of non-superreflexive spaces possessing quasi-greedy basis with  conditionality constants ``as large as possible."  
Our goal in this article is to fill this gap. To that end we enhance and exploit a combination of  techniques developed independently, on the one hand by Garrig\'os and Wojtaszczyk in \cite{GW2014}, and, on the other hand, by Dilworth et al. in \cite{DKK2003},  and craft a wealth of new examples of non-superreflexive classical Banach  spaces having quasi-greedy bases $\BB$ with $k_{m}[\BB]=\OO(\log m)$.\end{abstract}

\thanks{Both authors partially supported by the Spanish Research Grant \textit{An\'alisis Vectorial, Multilineal y Aplicaciones}, reference number MTM2014-53009-P. F. Albiac also acknowledges the support of Spanish Research Grant \textit{ Operators, lattices, and structure of Banach spaces},  with reference MTM2016-76808-P}

\maketitle

\section{Introduction}\label{Introduction}

\noindent Let $\BB=(\xx_n)_{n=1}^\infty$ be a  basis for a Banach space $\XX$ and let $(\xx_n^*)_{n=1}^\infty$ be its sequence of coordinate functionals. Given  a subset $A$ of $\NN$, the \textit{coordinate projection} on $A$ is (when well defined) the linear operator
 \[
 \textstyle
S_A[\BB]\colon \XX\to \XX, \quad f\mapsto \sum_{n\in A} \xx_n^*(f) \, \xx_n.
\]
The basis $\BB$  is \textit{unconditional} if and only if
 $\sup_{A \text{ finite}} \Vert S_A[\BB]\Vert<\infty$. Thus, in some sense, the \textit{conditionality} of   $\BB$ can be measured in terms of the growth of the sequence 
\[
\textstyle
k_m[\BB]:=\sup_{|A|\le m} \Vert S_A[\BB]\Vert, \quad m\in\NN.
\]
Recall that a  basis $\BB=(\xx_j)_{j=1}^\infty$   is said to be {\it quasi-greedy} if it is semi-normalized  (i.e., $\textstyle 0<\inf_{j\in\NN} \Vert \xx_j\Vert \le \sup_{j\in\NN} \Vert \xx_j\Vert<\infty$) and 
 there is a constant $C$ such that 
 \begin{equation}\label{DefAG}
\Vert f - S_A[\BB](f)\Vert \le C \Vert f \Vert
\end{equation}
whenever $f\in \XX$ and  $A\subseteq \NN$ are such that $|\xx_j^*(f)|\le |\xx_k^*(f)|$ for all $j\in \NN\setminus A$ and all $k\in A$.  The least constant $C$ such that \eqref{DefAG} holds  is known as the quasi-greedy constant of the basis (see \cite{AA2017bis}*{Remark 4.2}).

The next theorem summarizes the connection that exists between  superreflexivity  and the conditionality constants of quasi-greedy  bases. 
Recall that a space $\XX$ is said to be {\it superreflexive} if every Banach space finitely representable in $\XX$ is reflexive.
\begin{theorem}[see \cites{DKK2003,GW2014,AAGHR2015,AAW2017}]\label{CharacterizationSR}  Let $\XX$ be a Banach space.
\begin{enumerate}
\item[(a)] If $\BB$ is a quasi-greedy basis for $\XX$ then $k_m[\BB]\lesssim \log m$ for $m\ge 2$. 

\item[(b)] If $\XX$ is non-superreflexive then there is a quasi-greedy basis $\BB$ for a Banach space finitely representable in $\XX$ with $k_m[\BB] \gtrsim \log m $  for $m\ge 2$.

\item[(c)] If $\XX$ is superreflexive and $\BB$ is quasi-greedy then there is 
$0<a<1$ such that  $k_m[\BB]\lesssim (\log m)^a$ for $m\ge 2$. 

\item[(d)] For every $0<a<1$ there is a quasi-greedy basis $\BB$ for a Banach space (namely, $\ell_2$) finitely representable in $\XX$  with  $k_m[\BB] \gtrsim (\log m)^a$ for $m\ge 2$.
\end{enumerate}
\end{theorem}
Theorem~\ref{CharacterizationSR}  characterizes both the superreflexivity and the lack of superreflexivity of a Banach space $\XX$ in terms of the growth of the conditionality constants of quasi-greedy bases. It could be argued that   the quasi-greedy bases whose existence is guaranteed in parts (b) and (d) lie outside  the space $\XX$, and that,  although this approach is consistent when dealing with ``super'' properties, in truth it does  not  tackle the question of the existence of a quasi-greedy basis  with large conditionality constants in the space $\XX$ itself! Hence this discussion naturally leads to the following two questions:

 \noindent\textbf{Question A}. Let $\XX$ be a non-superreflexive Banach space with a basis. Is there a quasi-greedy basis $\BB$ for $\XX$ with $k_m[\BB]\approx \log m$ for $m\ge 2$?

 \noindent\textbf{Question B}. Let $\XX$ be a superreflexive Banach space with a basis.   Given $0<a<1$, is there a quasi-greedy basis $\BB$ for $\XX$ with $k_m[\BB]\gtrsim (\log m)^a$ for $m\ge 2$?
 
Questions A and B can be regarded as a development of the query initiated by Konyagin and Telmyakov in 1999 \cite{KonyaginTemlyakov1999} of finding conditional quasi-greedy bases in general Banach spaces, and which has evolved  towards the more specific quest  of finding quasi-greedy bases  ``as conditional as possible.''  The reader will find a detailed account of this process in  the papers \cites{Wo2000,DKW2002,DKK2003,Gogyan2010, GHO2013, AAW2017}.

Let us outline the state of the art of those two questions. Garrig\'os and Wojtaszczyk  proved in \cite{GW2014}   that Question B has a positive answer for  $\XX=\ell_p$, $1<p<\infty$. As for Question A, it is known that   Lindenstrauss' basic sequence in $\ell_1$, the Haar system  in $\BV(\RR^d)$ for $d\ge 2$,  and  the unit-vector system in the Konyagin-Telmyakov space  $KT(\infty,p)$ for $1<p<\infty$, are all quasi-greedy basic sequences with conditionality constants as large as possible (see \cites{GHO2013,BBGHO2018}). Moreover, in \cite{GW2014} it is proved that the answer to Question A is positive por $\ell_1\oplus\ell_2\oplus c_0$, and in \cite{AAGHR2015} that  the same holds true for mixed-norm spaces  of the form $(\bigoplus_{n=1}^\infty \ell_1^n)_q$ ($1<q<\infty$), providing this way the first-known examples of reflexive Banach spaces having quasi-greedy bases with conditionality constants as large as possible. More recently, the authors constructed in \cite{AAW2017}  the first-known examples of Banach spaces  of nontrivial type and nontrivial cotype for which the answer to Question A is positive. These   spaces are denoted by  $ \PX_{p,q}^0\oplus\PX_{p,q}^0\oplus\ell_2$, $1<p,q<\infty$,   where $\PX^0_{p,q}$ and  $\PX_{p,q}$  are  the interpolation spaces
 \begin{equation}\label{PXSpaces}
  \PX^0_{p,q}=(v_1^0,c_0)_{\theta,q}, \quad\text{and}\quad \PX_{p,q}=(v_1,\ell_\infty)_{\theta,q}, \quad p=1/(1-\theta),
  \end{equation}
   defined from the space of sequences of bounded variation
\[
\textstyle
v_1=\{ (a_j)_{j=1}^\infty \colon |a_1|+\sum_{j=2}^\infty |a_{j}-a_{j-1}|<\infty\},
\]
and the subspace $v_1^0$ of $v_1$ resulting from  the intersection of $v_{1}$ with $c_{0}$. 
Here and throughout this paper,  $(\XX_0,\XX_1)_{\theta,q}$ denotes the Banach space obtained by applying the real interpolation method  to the Banach couple $(\XX_0,\XX_1)$ with indices $\theta$ and $q$. 
 
 In this article we develop the necessary  machinery that permits to extend  the scant list of known Banach spaces for which the answer either to Question A or to Question B is positive. Moreover, in some cases, the examples of bases we provide  are not only quasi-greedy but are  \textit{almost greedy}. Recall that a basis $\BB=(\xx_j)_{j=1}^\infty$ for a Banach space $\XX$ is \textit{almost greedy} if there is a constant $C$ such that
\[
\Vert f - S_A[\BB](f)\Vert \le C \Vert f - S_B[\BB](f)\Vert
\]
whenever $f\in\XX$, $|B|\le |A|<\infty$, and $|\xx_j^*(f)| \le |\xx_k^*(f)|$ for any $j\in \NN\setminus A$ and $k\in A$. Almost greedy bases were characterized in  \cite{DKKT2003} as those bases that are simultaneously quasi-greedy and democratic. 

 Our study includes, among other spaces, the finite direct sums  
 \[
 D_{p,q}:=\begin{cases} 
 \ell_p\oplus\ell_q  \text{ if }1\le p,q<\infty,  \\  
 \ell_p\oplus c_0 \text{  if } 1\le p<\infty \text{  and  } q=0,
  \end{cases}
 \]
 the matrix spaces 
 \[
 Z_{p,q}:=\begin{cases}
  \ell_q(\ell_p) \text{  if  }1\le p,q<\infty,  \\  
  c_0(\ell_p)  \text{  if  } 1\le p<\infty \text{  and  } q=0, \\
 \ell_q(c_0) \text{  if  } p=0 \text{  and  } 1\le q<\infty, 
 \end{cases}
 \]
 and the mixed-norm spaces of the family
 \[
B_{p,q}:=(\oplus_{n=1}^\infty \ell_p^{n})_q, \,  p\in[1,\infty],  \, q\in\{0\}\cup[1,\infty), \,q\not=0 \text{ when }p=\infty.
\]
We use  $(\oplus_{n=1}^\infty \XX_n)_q$ to denote the direct sum of the Banach spaces $\XX_n$ in the $\ell_q$-sense ($c_0$-sense if $q=0$).

In Section~\ref{Main}, roughly speaking, we prove that in all non-superreflexive spaces in the aforementioned list the answer  to Question A is positive. We also show that, in turn, in all superreflexive spaces in the above list the answer to Question B is positive. Previously, in Section~\ref{Preliminaries}, we introduce the tools that we will use to achieve this goal.

Throughout this article we follow standard Banach space terminology and notation as can be found, e.g.,  in \cite{AlbiacKalton2016} but we would like to single out the notation that  is more commonly employed. We deal with real or complex Banach spaces, and $\FF$ will  denote the underlying scalar field.  As it is customary,  we put $\delta_{j,k}=1$ if $j=k$ and $\delta_{j,k}=0$ otherwise. Given $j\in\NN$, the $j$-th unit vector is defined by $\ee_j=(\delta_{j,k})_{k=1}^\infty$ and  the {\it unit-vector system} will be sequence $(\ee_j)_{j=1}^\infty$. Given families of positive real numbers $(\alpha_i)_{i\in I}$ and $(\beta_i)_{i\in I}$, the symbol $\alpha_i\lesssim \beta_i$ for $i\in I$ means that $\sup_{i\in I}\alpha_i/\beta_i <\infty$, while $\alpha_i\approx \beta_i$ for $i\in I$ means that $\alpha_i\lesssim \beta_i$ and $\beta_i\lesssim \alpha_i$ for $i\in I$. Applied to Banach spaces, the symbol $\XX\approx \YY$ means that the spaces $\XX$ and $\YY$ are isomorphic, while the symbol $\XX\lesssim_c\YY$ means that  $\XX$ is isomorphic to a complemented subspace of $\YY$. Given families of Banach spaces
$(\XX_i)_{i\in I}$ and $(\YY_i)_{i\in I}$, the symbol $\XX_i\lesssim_c \YY_i$ for $i\in I$ means that  the spaces
$\XX_i$ are uniformly isomorphic to  complemented subspaces of $\YY_i$, i.e., there are  linear operators
$L_i\colon \XX_i \to \YY_i$, $R_i\colon \YY_i \to \XX_i$ such that $R_i\circ L_i=\Id_{\XX_i}$ and $\sup_i \Vert R_i\Vert \Vert L_i\Vert<\infty$. Similarly the symbol $\XX_i\approx \YY_i$ for $i\in I$ means that Banach-Mazur distance from $\XX_i$ to $\YY_i$ is uniformly bounded. We write $\XX\oplus\YY$ for the Cartesian product of the Banach spaces $\XX$ and $\YY$ endowed with the norm 
\[
\Vert (x,y)\Vert=\max\{ \Vert x\Vert, \Vert y\Vert\}, \quad x\in\XX,\ y\in\YY. 
 \]
 The closed linear span of a family $(x_i)_{i\in I}$ in a Banach space will be denoted by $[x_i \colon i\in I]$. A basis always will be a Schauder basis. Given a basis $\BB=(\xx_j)_{j=1}^\infty$ for a Banach space $\XX$ and $m\in\NN$ we denote by $S_m[\BB]$ the $m$-th partial sum projection, i.e.,
\[
S_m[\BB](f)=\sum_{j=1}^m \xx_j^*(f) \, \xx_j, \quad f\in\XX.
\]
The \textit{fundamental function} (or \textit{upper democracy function}) of  a basis $\BB$ is the sequence
\[
\textstyle
\phi_m[\BB]=\sup_{|A|\le m}\Vert \sum_{j\in A} \xx_j\Vert
\]
If there is a constant $C$ such that $\phi_m[\BB]\le C \sum_{j\in A} \Vert \sum_{j\in A} \xx_j\Vert$ whenever  $m\le |A|$ the basis is said to be \textit{democratic}.

For $1\le p\le\infty$ and $d\in\NN$,  the finite-dimensional space $\ell_p^d$, will be the subspace of $\ell_p$ consisting of all sequences $(a_j)_{j=1}^\infty$ with $a_j=0$ for $j>d$. More generally, given a basis $\BB=(\xx_n)_{n=1}^\infty$ for a Banach space $\XX$ and  $d\in\NN$ we will consider the closed linear span of the truncated finite sequence $(x_{j})_{j=1}^{d}$, i.e.,
\[
\XX^d[\BB]=[\xx_j \colon 1\le j \le d].
\]

Other more specific notation will be introduced on the spot when needed.

\section{Background and preliminary results}\label{Preliminaries}

\noindent Most of the ideas behind the results we include in this preliminary section have appeared more or less explicitly  in the literature before. Nonetheless, for the sake of clarity and completeness, we shall include the statements of the results we need in the form that best suits our purposes and  the sketches of their proofs. 

\begin{definition} Given a basis $\BB$ for a Banach space $\XX$,
we define the {\it w-conditionality constants}  of $\BB$ by
\begin{equation*}
L_m[\BB]=\sup\left\{ \frac{ \Vert S_A[\BB](f)\Vert}{ \Vert f \Vert } \colon \max(\supp(f))\le m, \, A\subseteq\NN\right\},\qquad m\in \NN.
\end{equation*}
\end{definition}
Notice that $\BB$ is unconditional if and only if $\sup_m L_m[\BB]<\infty$. Hence, the growth of the sequence $(L_m[\BB])_{m=1}^{\infty}$ provides also a measure of the conditionality of the basis. Since $L_m[\BB]\le k_m[\BB]$ for all $m\in\NN$, any  result establishing that the size of the members of the sequence $(L_m[\BB])_{m=1}^\infty$ is large, is   a stronger statement than the corresponding one enunciated in terms of 
$(k_m[\BB])_{m=1}^\infty$. The papers \cites{AAW2017,GW2014} draw attention to the fact that, in some cases, the w-conditionality constants are more fit than the ``usual'' conditionality constants for transferring conditionality properties from a given basis to a basis constructed from it. This is the reason why we establish all the instrumental results of this section  in terms of  w-conditionality constants. Notice that, in contrast  to $(k_m[\BB])_{m=1}^\infty$,  the sequence $(L_m[\BB])_{m=1}^\infty$ is not necessarily doubling. This fact  leads us to use a doubling function in our statements. 
Recall that a function $\delta\colon[0,\infty)\to[0,\infty)$  is said to be \textit{doubling} if for some non-negative constant $C$ one has $\delta(2t)\le C \delta(t)$ for all $t\ge 0$. 

Given sequences $\BB_0=(\xx_j)_{j=1}^\infty$ and $\BB_1=(\yy_j)_{j=1}^\infty$ in Banach spaces $\XX$ and $\YY$, respectively,  their \textit{direct sum} $\BB_0 \oplus \BB_1$ will be the sequence in $\XX\times\YY$ given by
\begin{equation*}
\BB_0 \oplus \BB_1=((\xx_1,0),(0,\yy_1),(\xx_2,0),(0,\yy_2),\dots,((\xx_j,0),(0,\yy_j),\dots).
\end{equation*}
\begin{lemma}
[cf. \cite{GHO2013}]
\label{LemmaOne} Let  $\delta\colon[0,\infty)\to[0,\infty)$ be a  doubling increasing  function. Suppose that $\BB_0=(\xx_j)_{j=1}^\infty$ and $\BB_1=(\yy_j)_{j=1}^\infty$ are bases in Banach spaces $\XX$ and $\YY$, respectively, and assume that with $L_m[\BB_0]\gtrsim \delta(m)$ for $m\in\NN$. Then $\BB_1 \oplus \BB_2$ is a basis of $\XX\oplus\YY$ fulfilling 
$L_m[\BB_0\oplus\BB_1]\gtrsim \delta(m)$ for $m\in\NN$. Moreover we have:
 \begin{enumerate} 
 \item[(a)] If $\BB_0$ and $\BB_1$ are quasi-greedy so is $\BB_0\oplus\BB_1$;
 \item[(b)] If $\BB_0$ and $\BB_1$ are  democratic bases with $\phi_m[\BB_0] \approx\phi_m[\BB_1]$ for $m\in\NN$ then $\BB_0\oplus\BB_1$ is democratic with $\phi_m[\BB_0\oplus\BB_1]\approx\phi_m[\BB_0] \approx\phi_m[\BB_1]$
for $m\in\NN$.
\end{enumerate}\end{lemma}

\begin{proof} It is similar to the proof of \cite{GHO2013}*{Proposition 6.1}, so we omit it. \end{proof}

Our next lemma follows an idea from \cite{Wo2000} for constructing quasi-greedy bases. To state it properly it will be convenient to introduce some additional notation. 

Let $\BB= (\xx_j)_{j=1}^\infty$ be a basis in a Banach space $\XX$. Given a sequence of positive integers  $(d_n)_{n=1}^\infty$ we define a sequence $(\zz_k)_{k=1}^\infty$ that we denote $\bigoplus_{n=1}^\infty \BB[d_n]$ in the space $\XX^\NN$    by 
\begin{equation*}
\zz_k=(\underbrace{0,\dots,0}_{r-1 \text{ times}},\xx_j,0,\dots,0,\dots),
\end{equation*}
where, for a given $k\in\NN$, the integers $r$ and $j$ are univocally determined by the relations $k=j+\sum_{n=1}^{r-1} d_n$ and $1\le j\le d_n$.

\begin{lemma}\label{LemmaTwo} Suppose $p\in \{0\}\cup[1,\infty)$. Let $(d_n)_{n=1}^\infty$ be a sequence of positive integers,  $\delta\colon[0,\infty)\to[0,\infty)$ be a  doubling increasing  function, and $\BB=(\xx_j)_{j=1}^\infty$ be a basis for a Banach space $\XX$.  Assume that $L_m[\BB]\gtrsim \delta(m)$ for $m\in\NN$ and that there is $D>1$ such that $d_n\approx D^n$ for $n\in\NN$.
Then:
\begin{enumerate} 
\item[(a)] $\BB_0=\bigoplus_{n=1}^\infty \BB[d_n]$ is a basis for  the Banach space $(\bigoplus_{n=1}^\infty \XX^{d_n}[\BB])_p$ with 
$L_m[\BB_0] \gtrsim \delta(m)$ for $m\in\NN$. 
\item[(b)] If $\BB$ is quasi-greedy so is $\BB_0$.
\item[(c)] If $p\not=0$ and $\BB$ is democratic with $\phi_m[\BB]\approx m^{1/p}$ for $m\in\NN$ then $\BB_0$ is democratic with $\phi_m[\BB_0]\approx m^{1/p}$ for $m\in\NN$.
\end{enumerate}
 \end{lemma}

Lemma~\ref{LemmaTwo} is stated with sufficient generality to meet our goals in most cases. However, in one specific case
we will need a more  technical
 auxiliary result which includes Lemma~\ref{LemmaTwo} as a particular case. 
 
 Let $(d_n)_{n=1}^\infty$ be a sequence of positive integers. Suppose that $\BB=(\xx_j)_{j=1}^{\infty}$ is a basis in a Banach space $\XX$ and that, for each $n\in\NN$, $P_n$ and $Q_n$ are linear maps from $\XX^{d_n}[\BB]$  into the  (possibly zero) Banach spaces $\YY_n$ and $\ZZ_n$, respectively. We define a sequence $\bigoplus_{n=1}^\infty \BB[P_n,Q_n]=(\zz_k)_{k=1}^\infty$ in $\Pi_{n=1}^\infty \YY_n \times \Pi_{n=1}^\infty \ZZ_n$ by
\begin{align*}
\zz_k&=
((\underbrace{0,\dots,0}_{r-1 \text{ times}},P_n(\xx_j),0,\dots,0,\dots),
(\underbrace{0,\dots,0}_{r-1 \text{ times}},Q_n(\xx_j),0,\dots,0,\dots)),\\
&k=j+\sum_{n=1}^{r-1} d_n, \, 1\le j\le d_n.
\end{align*}
\begin{lemma}\label{LemmaThree} Suppose $p,q\in \{0\}\cup[1,\infty)$. Let $(d_n)_{n=1}^\infty$ be  a sequence of positive integers, $\delta\colon[0,\infty)\to[0,\infty)$ be a  doubling increasing  function, $(\YY_n)_{n=1}^\infty$ and $(\ZZ_n)_{n=1}^\infty$ be sequences of Banach spaces, 
 $\BB=(\xx_n)_{n=1}^\infty$ be a basis for a Banach space $\XX$, and $(P_n)_{n=1}^\infty$ and $(Q_n)_{n=1}^\infty$ be sequences of linear maps. Assume that 
\begin{itemize}
\item For each $n\in\NN$, $(P_n,Q_n)$ is an isomorphism from  $\XX^{d_n}[\BB]$ onto $\YY_n\oplus \ZZ_n$,
and $\sup_n \Vert (P_n,Q_n)\Vert  \Vert (P_n,Q_n)^{-1}\Vert<\infty$,
\item there is $D>0$ such that $d_n\approx D^n$ for $n\in\NN$, and
\item $L_m[\BB]\gtrsim \delta(m)$ for $m\in\NN$.
\end{itemize}
Then $\BB_0=\bigoplus_{n=1}^\infty \BB[P_n,Q_n]$ is a basis for  $(\bigoplus_{n=1}^\infty \YY_n)_p \oplus(\bigoplus_{n=1}^\infty \ZZ_n)_q$ with  $L_m[\BB_0] \gtrsim \delta(m)$. Moreover, if $\BB$ is quasi-greedy so is $\BB_0$.
\end{lemma}

\begin{proof}Let $C_1$ be such that $C_1 L_m[\BB]\ge  \delta(m)$ for all $m\in\NN$. Let $C_2$ and $C_3$ be such that $d_n\le C_2 D^n$ and $D^n \le C_3 d_n$ for all $n\in\NN$ and put $C_4=C_2 C_3 D^2/(D-1)$.  Since $\delta$ is doubling there a constant $C_5$ such that  $C_5 \delta(m)\ge \delta (C_4 m)$ for all $m\in\NN$. Replacing $\ell_0$ with $c_0$ and $\Vert\cdot\Vert_0$ with $\Vert \cdot\Vert_\infty$ when some of the indices involved is $0$, we consider the Banach space
\begin{align*}
\VV=\{(f_n)_{n=1}^\infty &\in  \Pi_{n=1}^\infty \XX^{d_n}[\BB] 
\colon   (\Vert P_n(f_n)\Vert)_{n=1}^\infty \in \ell_p, \, (\Vert Q_n(f_n)\Vert)_{n=1}^\infty \in \ell_q\},
\end{align*}
equipped with the norm 
\[
\Vert (f_n)_{n=1}^\infty\Vert =\max\{ \Vert (\Vert P_n(f_n)\Vert)_{n=1}^\infty\Vert_p, (\Vert P_n(f_n)\Vert)_{n=1}^\infty\Vert_q\}.
\]
Since the mapping
\[
\textstyle
f=(f_n)_{n=1}^\infty \mapsto ((P_n(f_n))_{n=1}^\infty, (Q_n(f_n))_{n=1}^\infty)
\]
is an isomorphism from $\VV$ onto  $(\bigoplus_{n=1}^\infty \YY_n)_p \oplus(\bigoplus_{n=1}^\infty \ZZ_n)_q$, it suffices to show that $\BB_1:=\bigoplus_{n=1}^\infty \BB[d_n]$ is a quasi-greedy basis for $\VV$ verifying
$L_m[\BB_1] \gtrsim \delta(m)$ for $m\in\NN$. It is clear that $\BB_1$ is a quasi-greedy basis with the same quasi-greedy constant as $\BB$, hence we must only take care of estimating its w-conditionality constants. 
Given $m\in\NN$, put $r\in\NN$ such that
$\sum_{n=1}^{r} d_n\le m <\sum_{n=1}^{r+1} d_n$. Since 
\[
m\le\sum_{n=1}^{r+1} d_n \le  C_2 \sum_{n=1}^{r+1} D^{r}=\frac{C_2 D}{D-1} (D^{r+1}-1) \le \frac{C_2 C_3 D^2}{D-1} d_r=C_4 d_r,
\]
we have
\[
C_1 C_5 L_m[\BB_1]\ge C_1 C_5 L_{d_r}[\BB]\ge  C_5\delta(d_r)\ge \delta(C_4 d_r)\ge  \delta(m),
\]
as desired. \end{proof}

Next, we appeal to a technique invented by Garrig\'os and Wojtaszczyk in \cite{GW2014}  for tailoring a basis $\GOW(\BB)$ for the direct sum $\XX\oplus\ell_2$ from a basis $\BB$ of  a Banach space of $\XX$.  This method, called for short the \textit{GOW-method} in \cite{AAW2017}, can be  summarized as follows.

\begin{theorem}[\cite{GW2014}]\label{ConditionalityGOW} Let  $\delta\colon[0,\infty)\to[0,\infty)$ be a  doubling increasing  function. Suppose  that $\BB$  is a basis of a Banach space $\XX$ such that $L_m[\BB] \ge \delta(m) $ for $m\in\NN$. Then
\begin{enumerate}
\item[(i)] The basis $\BB_0:=\GOW(\BB)$ is an almost greedy basis of $\YY:=\XX\oplus\ell_2$,
\item[(ii)]  $\phi_m[\BB_0]\approx m^{1/2}$,  and 
\item[(iii)] $L_m[\BB_0] \gtrsim  \delta(\log m)$ for all $m\in\NN$. 
\end{enumerate}
Moreover 
\[
\YY^{2^n-2}[\BB_0] =\XX^n[\BB] \oplus \ell_2^{2^n-n-2}, \quad n\in\NN.
\]
\end{theorem} 

\begin{remark} Notice that combining Proposition 4.4 and Theorem 5.4 from \cite{DKKT2003} with Theorem~\ref{ConditionalityGOW} yields that the dual basic sequence of the one obtained by the GOW-method is also almost greedy, with fundamental function of the order of $(m^{1/2})_{m=1}^\infty$.
\end{remark}
Garrig\'os and Wojtaszczy \cite{GW2014} combined Theorem~\ref{ConditionalityGOW}  with the next theorem  for building  quasi-greedy bases as conditional as possible in $\ell_2$. 
\begin{theorem}[cf. \cite{GW2014}*{Proposition 3.10}]\label{GWHilbert} For each $0<a<1$ there is a  basis $\BB$ in $\ell_2$ with $L_m[\BB]\gtrsim m^a$ for $m\in\NN$.
\end{theorem}
Another technique for tailoring quasi-greedy bases was developed by Dilworth et al.\  in \cite{DKK2003}. Among the many   important results contained in that work, the authors proved the following two. \begin{theorem}[see \cite{DKK2003}, Corollary 6.3 and Theorem 7.1]\label{DKKAG} Assume that a Banach space $\XX$ has a basis. Then $\XX\oplus \ell_1$ has an almost greedy basis $\BB$ such that $\phi_m[\BB] \approx m$ for $m\in\NN$.
\end{theorem}

\begin{theorem}[see \cite{DKK2003},  Theorem 7.2]\label{DKKQG} Assume that a Banach space $\XX$ has a basis and that
$1<p<\infty$ Then $\XX \oplus \ell_p $ has a quasi-greedy basis.
\end{theorem}

Next we summarize the information we need on the Lindenstrauss sequence $\LL=(\li_j)_{j=1}^\infty$, defined by
\[
 \li_j=\ee_j - \frac{1}{2}(\ee_{2j}+\ee_{2j+1}), \quad j\in\NN.
\]

\begin{theorem}[see \cites{Singer1970,DilworthMitra2001,LinPel1968,GHO2013}]\label{Linde}The Lindenstrauss sequence $\LL$ is an almost greedy basic sequence in $\ell_1$ verifying
\begin{itemize}
\item[(a)] $L_m[\LL]\approx \log m$ for $m\ge 2$, 
\item[(b)] $\ell_1^d[\LL] \approx \ell_1^d$ for $d\in\NN$, and
\item[(c)] $\phi_m[\LL] \approx m$ for $m\in\NN$.
\end{itemize}
\end{theorem}

\begin{proof}That $\LL$ is a basic sequence is proved in \cite{Singer1970}*{p.\ 455}, and that is quasi-greedy is showed in \cite{DilworthMitra2001}. The argument used in \cite{GHO2013}*{Example 1} for computing the conditionality constants of $\LL$  gives (a). Finally, (b) is obtained in \cite{LinPel1968}*{Example 8.1}, and (c) is a consequence of (b).
\end{proof}

Next, we consider the   \textit{summing system} $\SSS=(\sss_j)_{j=1}^\infty$, given by
\[
\textstyle
\quad \sss_j=\sum_{k=1}^j \ee_j.
\]
It is known that $\SSS$ is a conditional basis for $c_0$. Most proofs  of this fact 
(see, e.g., \cite{AlbiacKalton2016}*{Example 3.1.2}) give the following.
\begin{lemma}\label{coLemma}\label{Summingc0} The summing system $\SSS$ is a basis  for $c_0$ with  $L_m[\SSS]\approx m$ for $m\in\NN$, and $c_0^d[\SSS]=\ell_\infty^d$ for all $d\in\NN$.
\end{lemma}
By duality, the \textit{difference system} $\DD=(\dd_j)_{j=1}^\infty$, defined by
\[
\dd_j=\ee_j -\ee_{j-1}  \text{ (with the convention $\ee_0=0$)},
\]
is a conditional basis for $\ell_1$. Indeed, we have the following.
\begin{lemma}\label{l1Lemma} The difference system $\DD$ is a basis of $\ell_1$ such that $L_m[\DD]\approx m$ for $m\in\NN$, and  $\ell_1^d[\DD]=\ell_1^d$ for all $d\in\NN$.
\end{lemma}

A basis $(\xx_j)_{j=1}^\infty$ is said  to be of \textit{type $P$} if   $\sup_k \Vert \sum_{j=1}^k \xx_j \Vert <\infty$ and  $\inf_j \Vert \xx_j\Vert>0$. Notice that both the unit vectors system in $c_0$ and the difference system in $\ell_1$ are bases of type P.
The following lemma shows  that Banach spaces with a  basis of type P follow the pattern  of  $c_0$ and $\ell_1$ exhibited, respectively, in Lemmas~\ref{coLemma} and \ref{l1Lemma}.
\begin{lemma}[see \cite{AAW2017}]\label{AnsoLemma}Let $\BB$ be a basis of type $P$ of a Banach space $\XX$. Then there is a basis
$\BB_0$ for $\XX$ such that $L_m[\BB_0]\approx m$ for $m\in\NN$ and 
$\XX^{2^n-2}[\BB_0]=\XX^{2^n-2}[\BB]$ for all $n\ge 2$.
\end{lemma}
\begin{proof} The proof of \cite{AAW2017}*{Theorem 3.3}  gives the result, although is not  explicitly stated.
\end{proof}

The spaces $\PX^0_{p,q}$  and $\PX_{p,q}$ ($1<p<\infty$, $1\le q <\infty$) defined  in \eqref{PXSpaces} were introduced and studied by Pisier and Xu \cite{PisierXu1987}.  It is verified that  
$\PX^0_{p,q} \approx \PX_{p,q}$. Moreover, when $q>1$  these spaces have nontrivial type and nontrivial cotype  and  they are pseudo-reflexive. 

Our next proposition is a new addition  to the study of Pisier-Xu spaces, which will be used below. Recall that given $1\le q<\infty$ and a scalar sequence $\ww=(w_n)_{n=1}^\infty$, the  Lorentz sequence space $d_q(\ww)$ consists of all sequences $f$ in $c_{0}$ whose non-increasing rearrangement $(a_n)_{n=1}^\infty$ verifies
 \[
\Vert f \Vert_{d_q(\ww)}= \left( \sum_{n=1}^\infty a_n^q w_n \right)^{1/q} <\infty.
\]
In  the case when $\ww=(n ^{q/p-1})_{n=1}^\infty$  for some $1\le p<\infty$ we have that $\ell_{p,q}:=d_q(\ww)$ is the classical sequence Lorentz space of indices $p$ and $q$.

\begin{proposition}\label{PXLorentz}Let $1<p<\infty$ and $1\le q <\infty$.
Then $\ell_{p,q}\lesssim_c \PX_{p,q}^0$. Indeed, $(\ee_{2j-1})_{j=1}^\infty$ is a complemented basic sequence
isometrically equivalent to the unit vector system in $\ell_{p,q}$.
\end{proposition}
\begin{proof} Put $p=1/(1-\theta)$. Consider the linear maps $L,R\colon \FF^\NN\to \FF^\NN$ defined by
\begin{align}
L((a_j)_{j=1}^\infty)&=(a_1,0,a_2,0,\dots,a_j,0,a_{j+1},0,\dots), \label{Lifting}\\
R((a_j)_{j=1}^\infty)&=(a_1-a_2,a_3-a_4,\dots, a_{2j-1}-a_{2j}, \dots).\label{Retraction}
\end{align}
We have $\Vert L \colon  \ell_1 \to v_1^0\Vert\le 1, \Vert L \colon c_0 \to c_0\Vert\le 1, \Vert R \colon v_1^0 \to \ell_1\Vert \le 1$, and
$\Vert R \colon  c_0 \to c_0\Vert \le 2$. Taking into account that \[
(\ell_1,c_0)_{\theta,q}=(\ell_1,\ell_\infty)_{\theta,q}=\ell_{p,q}\] (see, e.g., \cite{BenSha1988}*{Theorem 1.9}),   interpolation gives
$\Vert L \colon \ell_{p,q} \to \PX^0_{p,q} \Vert\le 1$ and $\Vert R \colon \PX^0_{p,q}  \to \ell_{p,q} \Vert\le 2^{\theta}$.
Since $R (L(f))=f$ for every $f\in\FF^\NN$ we are done.
\end{proof}

\begin{corollary} Let $1<p<\infty$ and $1\le q <\infty$. Then
$\ell_q \lesssim_c \PX^0_{p,q}$.
\end{corollary}
\begin{proof}In light of Proposition~\ref{PXLorentz}, it suffices to see that $\ell_q\lesssim_c \ell_{p,q}$.
 By \cite{LinTza1972}*{Proposition 4} we have $\ell_q\lesssim_c \ell_{p,q}$
if $q\le p$ and  $\ell_{q'} \lesssim_c \ell_{p',q'}$ otherwise. We conclude the proof by dualizing  (see \cite{Allen1978}*{Theorem 1}).
\end{proof}

\section{Banach spaces having quasi-greedy bases with large conditionality constants}
\label{Main}

\noindent The common thread running through this section is the search for results that will allow us to include the spaces $Z_{p,q}$, $B_{p,q}$, and $D_{p,q}$ (see Section~\ref{Introduction}) in the list of Banach spaces possessing highly conditional quasi-greedy bases.
We recall that the matrix spaces $Z_{p,q}$ are isomorphic to Besov spaces over Euclidean spaces (see, e.g., \cite{AA2016}) and that the  mixed-norm spaces $B_{p,q}$ are isomorphic to Besov spaces over the unit interval (see, e.g., \cite{AA2017}*{Appendix 4.2}).

Apart from the trivial cases,  namely 
\[D_{q,q}\approx Z_{q,q}\approx B_{q,q}\approx \ell_q, \quad  1\le q<\infty,\] and the case \begin{equation}\label{Peuchinski}
\ell_q \approx B_{2,q}, \quad 1<q<\infty,
\end{equation}
all the above-mentioned    spaces are mutually non-isomorphic (see \cite{AA2017}).
 The isomorphism in  \eqref{Peuchinski} was obtained by Pe{\l}czy{\'n}ski  in \cite{Pel1960}  by combining the uniform complemented embeddings
\begin{equation}\label{PeuRad}
\ell_2^n \lesssim_c \ell_p^{2^n}  \text{ for } n\in\NN \text{, if } 1<p<\infty,
\end{equation}
 (which can  be obtained as a consequence of the boundedness of the Rademacher projections in  $L_p$) 
 with the Pe{\l}czy{\'n}ski  decomposition technique (see, e.g., \cite{AlbiacKalton2016}*{Theorem 2.2.3}).
Another well-known consequence  of Pe{\l}czy{\'n}ski  decomposition technique, 
is that for any unbounded sequence of integers $(d_n)_{n=1}^\infty$   we have
\begin{equation}\label{BesovIso}
B_{p,q}\approx(\oplus_{n=1}^\infty \ell_p^{d_n})_q, \quad p\in [1,\infty], \, q\in\{0\}\cup[1,\infty),
\end{equation}
(see, e.g.,  \cite{AA2017}*{Appendix 4.1}.)

First  we deal with Banach spaces of trivial  type. For that it is crucial to know  how $\ell_1$ is positioned inside the spaces.

\begin{theorem}\label{T1}Let $\XX$ be a Banach space with a basis. If $\ell_1\lesssim_c \XX$ then
$\XX$ has an almost greedy basis  $\BB$ with $\phi_m[\BB]\approx m$ for $m\in\NN$, and $L_m[\BB] \approx \log m$ for $m\ge 2$.
 \end{theorem}

\begin{theorem}\label{T3}Let $\XX$ be a Banach space with a quasi-greedy basis. If
 $B_{1,0}\lesssim_c\XX$
then there is a quasi-greedy basis $\BB$ for $\XX$ with $L_m[\BB] \approx \log m$ for $m\ge 2$.
\end{theorem}

\begin{proof}[Proof of Theorems~\ref{T1} and \ref{T3}] Let $p\in \{0,1\}$.
 For  $p=0$, assume that the hypotheses of Theorem~\ref{T3} hold (respectively, assume that for $p=1$ the hypotheses of Theorem~\ref{T1} hold).

By Lemma~\ref{LemmaTwo} and Theorem~\ref{Linde} 
$\BB_0=\bigoplus_{n=1}^\infty  \LL[2^n]$ is a quasi-greedy basis for $\YY:=(\bigoplus_{n=1}^\infty \ell_1^{2^n}[\LL])_p$ such that $L_m[\BB_0]\approx \log m$ for $m\ge 2$. Moreover, in the case when $p=1$, $\BB_0$ is democratic with 
$\phi_m[\BB_0]\approx m$.  

In the case when $p=0$ the hipothesis give a quasi-greedy basis $\BB_1$ for $\SSB:=\XX$ (respectively,  when $p=1$ Theorem~\ref{DKKAG} gives a quasi-greedy basis $\BB_1$ for $\SSB:=\XX\oplus\ell_1$ that is democratic with $\phi_m[\BB_1]\approx m$). 

 By Lemma~\ref{LemmaOne}, $\BB_2:=\BB_1\oplus\BB_0$ is a quasi-greedy basis for $\ZZ:= \SSB\oplus\YY$ such that 
 $L_m[\BB_2]\approx \log m$. Moreover, in the case when $p=1$, $\BB_2$ is democratic with $\phi_m[\BB_2]\approx m$. 
Combining Theorem~\ref{Linde} with the isomorphisms \eqref{BesovIso}, 
 $\ell_1\oplus B_{1,1}\approx B_{1,1}$ and  $ B_{1,p}\oplus B_{1,p}\approx B_{1,p}$  give
 \[
 \ZZ\approx \SSB\oplus \left(\bigoplus_{n=1}^\infty \ell_1^{2^n}\right)_p \approx \SSB\oplus B_{1,p}\approx
 \XX \oplus B_{1,p} \approx \XX,
 \]
and, so, the proof is over. \end{proof}

\begin{example}\
\begin{enumerate}
\item[(i)] The list of Banach spaces for which Theorem~\ref{T1} applies includes
 $D_{1,p}$, $Z_{p,1}$ and $Z_{1,p}$ for $p\in\{0\}\cup(1,\infty)$, 
$B_{p,1}$ for $p\in(1,\infty]$, $\ell_1$, $L_1[0,1]$, the Hardy space $H_1$, and the Lorentz sequence spaces $d_1(\ww)$ for $\ww$ decreasing.
\item[(ii)] By invoking \cite{LinPel1968}*{Proposition 7.3} and  \cite{JRZ1971}*{Theorem 5.1}, 
 Theorem~\ref{T1} applies to any separable $\SL_1$-space. Indeed, since $\SL_1$-spaces are GT-spaces (see \cite{LinPel1968}*{Theorem 4.1}), in light of \cite{DST2012}*{Theorem 4.2}, the conclusion on democracy is redundant for such spaces.
 
 \item[(iii)] Theorem~\ref{T3} moves the space   $B_{1,0}$   to the list of Banach spaces possessing a quasi-greedy basis as conditional as possible.
 \end{enumerate}
 \end{example}
 
 \begin{remark} We would like to point out that the argument used in the above proof also works for $p\in(1,\infty)$. However, we will use an alternative method for including the spaces $B_{1,p}$ is our list.
 \end{remark}
  
The next five theorems contain  our study of the case of Banach spaces with trivial cotype. We emphasize that  the lack of a conditional quasi-greedy basis in $c_0$ (see \cite{DKK2003}*{Corollary 8.6}) constitutes an added difficulty when tackling this task.

\begin{theorem}\label{T4} Let $\XX$ be a Banach space with a basis of type P. If $\ell_2\lesssim_c \XX$, then
 $\XX$ has an almost greedy $\BB$ with $\phi_m[\BB]\approx m^{1/2}$ for $m\in\NN$ and
$L_m[\BB]\approx \log m$ for $m\ge 2$.\end{theorem}

\begin{theorem}\label{T5}Let $\XX$ be a Banach space with a basis. Assume that $\ell_2\lesssim_c \XX$ and
either $c_0\lesssim_c \XX$ or $\ell_1\lesssim \XX$. Then $\XX$ has an almost greedy basis $\BB$ with $\phi_m[\BB]\approx m^{1/2}$ for $m\in\NN$ and $L_m[\BB]\approx \log m$ for $m\ge 2$.\end{theorem}

\begin{theorem}\label{T6}Let $\XX$ be a Banach space with a basis. Suppose that 
either $B_{\infty,2}\lesssim_c \XX$ or $B_{1,2}\lesssim \XX$. Then $\XX$ has an almost greedy $\BB$ with $\phi_m[\BB]\approx m^{1/2}$ for $m\in\NN$ and $L_m[\BB]\approx \log m$ for $m\ge 2$.\end{theorem}

\begin{theorem}\label{T7} Let $\XX$ be a Banach space with a basis and $1<p<\infty$. Assume that  either 
$B_{\infty,p}\lesssim_c \XX$ or $B_{1,p}\lesssim_c \XX$. Then $\XX$ has a quasi-greedy basis $\BB$ such that  $L_m[\BB]\approx \log m$ for $m\ge 2$.
\end{theorem}

\begin{theorem}\label{T8}Let $\XX$ be a Banach space with a quasi-greedy basis.  If  $B_{q,0}\lesssim_c \XX$ for some $1< q<\infty$, then $\XX$ has a quasi-greedy basis $\BB$ with $L_m[\BB]\approx \log m$ for $m\ge 2$.
\end{theorem}

\begin{proof}[Proof of Theorems~\ref{T4}, \ref{T5}, \ref{T6},  \ref{T7} and \ref{T8}] Let us consider Banach spaces $\YY$ and $\ZZ$ such that
\begin{itemize}
\item[(\textbf{A})] either $\YY=\{0\}$ and $\ZZ$ has basis $\BB$ of type P,
\item[(\textbf{B})] or $\YY$ is a Banach space with a basis and $\ZZ$ is either $c_0$ or $\ell_1$.
\end{itemize}

By Invoking  Lemmas~\ref{coLemma}, \ref{l1Lemma},  \ref{AnsoLemma} and \ref{LemmaOne} we claim that the space $\YY\oplus\ZZ$ has a basis $\BB_0$ with $L_m[\BB_0]\approx m$ for $m\in\NN$. 
From Theorem~\ref{ConditionalityGOW} we deduce that $\VV:=\YY\oplus\ZZ\oplus\ell_2$ has  an almost greedy basis $\BB_1=\GOW(\BB_0)$ with  $\phi_m[\BB_1]\approx m^{1/2}$ for $m\in\NN$ and $L_m[\BB_1] \approx \log m$ for $m\ge 2$. 

In order to prove Theorem~\ref{T4} we consider the case (A) with $\ZZ=\XX$. Notice that in this case, since
$\ell_2\oplus\ell_2=\ell_2$, we have that  $\VV=\XX\oplus\ell_2\approx \XX$.

In order to prove Theorem~\ref{T5} we consider the case (B) with $\YY=\XX$.  Now, since $\ZZ\oplus\ZZ\approx\ZZ$ also holds, we have that  $\VV=\XX\oplus\ZZ\oplus\ell_2\approx \XX$.

Assume that we are in  case (A) and  that $\ZZ$ is either $c_0$ (in which case we put $r=\infty$) or $\ell_1$
(in which case we put $r=1$). In this situation we choose $\BB_0$ to be  the summing system when $r=\infty$ and the difference system when $r=1$. Therefore
\[
\VV^{2^n-2}[\BB_1]=\ell_2^{2^n-n-2}\oplus \ZZ^n[\BB_0]=\ell_2^{2^n-n-2}\oplus \ell_r^n, \quad n\in\NN.
\]
Let $p\in\{0\}\cup(1,\infty)$ and assume that when $p=2$ the hypotheses of Theorem~\ref{T6} hold, that when $p=0$ the hypotheses of Theorem~\ref{T8} hold, and that when $p\in(1,2)\cup(2,\infty)$ the hypotheses of Theorem~\ref{T7} hold.
Applying Lemma~\ref{LemmaTwo} we get a quasi-greedy basis $\BB_2:=\bigoplus_{n=1}^\infty \BB_1[2^n-2]$  for the Banach space \[
\WW=\left( \bigoplus_{n=1}^\infty (\ell_2^{2^n-n-2}\oplus \ell_r^n)\right)_p
\]
with $L_m[\BB_2]\approx \log m$. Moreover, in the case when $p=2$, $\BB_2$ is democratic with
$\phi_m[\BB_2]\approx m^{1/2}$.

In the case when $p=0$, we have that $\SSB:=\XX$ has a quasi-greedy basis $\BB_3$.
In the case when $p\in(1,2)\cup(2,\infty)$  Theorem~\ref{DKKQG} yields the existence of a quasi-greedy basis $\BB_3$ for $\SSB:=\XX\oplus \ell_p$.
In the case when $p=2$ Theorem~\ref{ConditionalityGOW} gives that $\SSB:=\XX\oplus \ell_2$ has 
 a quasi-greedy basis $\BB_3$ which is democratic with $\phi_m[\BB_2]\approx m^{1/2}$.
Then, by Lemma~\ref{LemmaOne},  
$\BB_4:=\BB_3\oplus\BB_2$ is a quasi-greedy basis for the Banach space
$\UU:=\SSB\oplus\WW$ with $L_m[\BB_4]\approx \log m$. Moreover, in the case
$p=2$, $\BB_4$ is democratic with $\phi_m[\BB_4]\approx m^{1/2}$.

With the basis $\BB_4$ and the isomorphisms \eqref{Peuchinski} and \eqref{BesovIso} in hand, we are ready to complete the proof.
Let $p\in(1,\infty)$.
 Taking into account that
$B_{r,p}\oplus \ell_p\approx B_{r,p}$ and that $B_{r,p}\oplus B_{r,p}\approx B_{r,p}$, we obtain
\begin{align*}
\UU&
\approx
\XX\oplus \ell_p \oplus \left( \bigoplus_{n=1}^\infty \ell_2^{2^n-n-2} \right)_p
\oplus
\left( \bigoplus_{n=1}^\infty \ell_r^n \right)_p \\
&\approx
 \XX\oplus \ell_p \oplus B_{2,p}\oplus \left( \bigoplus_{n=1}^\infty \ell_r^n\right)_p \\\
&=\XX\oplus \ell_p \oplus \ell_p \oplus B_{r,p} \approx \XX.
\end{align*}
This completes the proof of Theorems~\ref{T6} and \ref{T7}. 

Let $p=0$. 
Notice that the relations \eqref{PeuRad} and \eqref{BesovIso} give the isomorphisms $c_0\oplus B_{2,0} \approx B_{2,0}\oplus B_{2,0} \approx B_{2,0}$ as well as the complemented embedding $B_{2,0}\lesssim_c B_{q,0}$. Consequently,
$B_{2,0}\lesssim_c \XX$.
The chain of isomorphisms
\[
\UU
\approx
\XX \oplus \left( \bigoplus_{n=1}^\infty \ell_2^{2^n-n-2} \right)_0
\oplus
\left( \bigoplus_{n=1}^\infty \ell_\infty^n \right)_0 
\approx
 \XX\oplus B_{2,0} \oplus c_0 
=\XX,
\]
completes the proof of Theorem~\ref{T8}.
\end{proof}

\begin{example}\
 \begin{enumerate}
\item[(i)] The  unit-vector system is a shrinking basis of type P for the James space $\JJ$ introduced in \cite{James1950} (see, e.g., \cite{AlbiacKalton2016}*{Proposition 3.4.4 and Remark 3.4.5}).  It is also known that $\ell_2\lesssim_c \JJ$; indeed, the linear operator $L$  defined  in \eqref{Lifting} is bounded from $\ell_2$ into $\JJ$ and the linear operator $R$ defined  in \eqref{Retraction} is bounded from $\JJ$ into $\ell_2$ (see also \cite{CLL1977}*{Corollary 11}). Hence Theorem~\ref{T4} applies to $\JJ$ and to $\JJ^*$. By Proposition~\ref{PXLorentz}  and  \cite{AAW2017}*{Proposition 2.10} (which states that the unit-vector system is a basis of type P for Pisier-Xu spaces),
Theorem~\ref{T4}  also applies to $\PX^0_{p,2}$ for
$1<p<\infty$.

\item[(ii)] Theorem~\ref{T5} applies to the spaces $D_{2,0}$, $D_{2,1}$, $Z_{1,2}$, $Z_{2,1}$, $Z_{2,0}$, $Z_{0,2}$,
 the Hardy space $H_1$ and its predual  $\VMO$. 
 
 \item[(iii)] Theorem~\ref{T6} applies to the spaces $B_{\infty,2}$ and $B_{1,2}$.
 
 \item[(iv)] Theorem~\ref{T7} applies to the spaces $B_{\infty,p}$, $B_{1,p}$ and $Z_{0,p}$ for $1<p<\infty$.

\item[(v)] Theorem~\ref{T8} applies to $B_{p,0}$ and $Z_{p,0}$  for $1<p<\infty$,.
\end{enumerate}
Let us mention that, for $p\not=2, $, $D_{p,0}$ is not in our list of Banach spaces with  a quasi-greedy basis as conditional as possible yet.
\end{example}

\begin{theorem}\label{T9}Let $\XX$ be a Banach space and $1< p<\infty$. Assume that
$\ell_p\lesssim_c \XX$, that $c_0\lesssim_c \XX$  and that $\XX$ has a  basis. Then $\XX$ has a quasi-greedy basis $\BB$ with $L_m[\BB]\approx \log m$ for $m\ge 2$.
\end{theorem}
\begin{proof}Since $\ell_p\oplus \ell_p\approx \ell_p$ and $c_0\oplus c_0\approx c_0$, taking into account Theorem~\ref{DKKQG} and Lemma~\ref{LemmaOne}, it sufices to prove the result for $\XX=D_{p,0}$. Let $\SSS$ the summing basis if $c_0$ and let $P_n$ and $Q_n$ be the canonical projections from  $\ell_\infty^n \oplus \ell_2^{2^n-n-2}$ onto, respectively, $\ell_\infty^n$ and $\ell_2^{2^n-n-2}$. From Lemma~\ref{coLemma}, Theorem~\ref{ConditionalityGOW}, and  Lemma~\ref{LemmaThree} we infer that $\BB_0=\bigoplus_{n=1}^\infty \GOW(\SSS)[P_n,Q_n]$ is a quasi-greedy basis for 
\[
\left(\bigoplus_{n=1}^\infty \ell_\infty^n\right)_0
\oplus \left(\bigoplus_{n=1}^\infty \ell_2^{2^n-n-2}\right)_p\approx c_0 \oplus B_{2,p}
\approx c_0\oplus \ell_p \approx D_{p,0}
\]
with $L_m[\BB_0]\approx \log m$ for $m\ge 2$.
\end{proof}

\begin{example}Theorem~\ref{T9} applies to $D_{p,0}$ for $1<p<\infty$.
\end{example}

To close this article, we revisit the advances in the superreflexive case carried out in \cite{GW2014}.

\begin{theorem}[cf.\ \cite{GW2014}*{Theorem 1.2 and Corollary 3.13}]\label{T30}
Let $\XX$ be a Banach space with a basis. If
$\ell_2\lesssim_c \XX$ then for any  $0<a<1$ the space $\XX$ has a almost greedy basis $\BB$ with
$\phi_m[\BB]\approx m^{1/2}$ and
 $L_m[\BB]\gtrsim(\log m)^a$ for $m\in\NN$.
\end{theorem}
\begin{proof}
Combining  Theorem~\ref{ConditionalityGOW} with Theorem~\ref{GWHilbert} we obtain an almost greedy basis $\BB_1$ for $\ell_2\oplus\ell_2$ with $\phi_m[\BB_1]\approx m^{1/2}$ and $L_m[\BB_1]\gtrsim (\log m)^a$. Applying again Theorem~\ref{ConditionalityGOW} we get an almost greedy basis $\BB_2$ for $\XX\oplus\ell_2$ with $\phi_m[\BB_2]\approx m^{1/2}$. Hence, by Lemma~\ref{LemmaOne}, $\BB_2\oplus\BB_1$ is a basis as desired for $\XX\oplus\ell_2\oplus\ell_2\oplus\ell_2\approx \XX$.
\end{proof}

\begin{example}\
\begin{enumerate}
\item[(i)] Theorem~\ref{T30} applies to the spaces $Z_{2,p}$, $Z_{p,2}$, $B_{p,2}$, $D_{p,2}$,  $L_p$  for $p\in(1,\infty)$, and  the Lorentz sequence spaces $\ell_{p,2}$ for $1< p<\infty$. 

\item[(ii)] More generally (see \cite{LinPel1968}*{Proposition 7.3},   \cite{JRZ1971}*{Theorem 5.1}, \cite{JohnsonOdell1974}*{Corollary 1}
and \cite{KP1962}*{Corollary 1}), Theorem~\ref{T30} applies 
 to any separable $\SL_p$-space  that is non-isomorphic to $\ell_p$, for $p\in(1,\infty)$.
 \end{enumerate}
\end{example}

\begin{theorem}[cf.\ \cite{GW2014}*{Corollary 3.12}]\label{T31}
Let $\XX$ be a Banach space with a basis. 
If $\ell_p\lesssim_c \XX$ for some $1< p<\infty$, then for any  $0<a<1$ the space $\XX$ has a quasi-greedy basis $\BB$ with $L_m[\BB]\gtrsim(\log m)^a$ for $m\in\NN$.
\end{theorem}

\begin{proof} Taking into account Lemma~\ref{LemmaOne},  Theorem~\ref{DKKQG}, and the fact that $\ell_p\oplus\ell_p\approx\ell_p$, it suffices to consider the case $\XX=\ell_p$.
Use Theorem~\ref{T30} to pick a quasi-greedy basis   $\BB$ for $\ell_2$ with $L_m[\BB]\gtrsim (\log m)^a$ for all $m\in\NN$. By Lemma~\ref{LemmaTwo}, $\BB_0=\bigoplus_{n=1}^\infty \BB[2^n]$ is a quasi-greedy basis for $\YY=(\bigoplus \ell_2^{2^n}[\BB])_p$ with $L_m[\BB_0]\gtrsim (\log m)^a$. Since any $d$-dimensional Hilbert space is isometric to $\ell_2^d$, the isomorphisms \eqref{Peuchinski} and \eqref{BesovIso} yield $\YY\approx (\bigoplus \ell_2^{2^n})_p\approx B_{2,p}\approx\ell_p$.
\end{proof}

\begin{example} Theorem~\ref{T31} applies to the spaces $\ell_p$, $D_{p,q}$, $B_{p,q}$ and $Z_{p,q}$
for $p,q\in(1,\infty)\setminus\{2\}$, and the Lorentz sequence spaces $\ell_{p,q}$ for $1< p,q<\infty$, $q\not=2$. 
\end{example}




\begin{bibsection}
\begin{biblist}


\bib{AA2016}{article}{
   author={Albiac, F.},
   author={Ansorena, J.~L.},
   title={The isomorphic classification of Besov spaces over $\mathbb{R}^d$
   revisited},
   journal={Banach J. Math. Anal.},
   volume={10},
   date={2016},
   number={1},
   pages={108--119},
}


\bib{AA2017}{article}{
   author={Albiac, F.},
   author={Ansorena, J.L.},
   title={Isomorphic classification of mixed sequence spaces and of Besov
   spaces over $[0,1]^d$},
   journal={Math. Nachr.},
   volume={290},
   date={2017},
   number={8-9},
   pages={1177--1186},
}

\bib{AA2017bis}{article}{
   author={Albiac, F.},
   author={Ansorena, J. L.},
   title={Characterization of 1-almost greedy bases},
   journal={Rev. Mat. Complut.},
   volume={30},
   date={2017},
   number={1},
   pages={13--24},
}



\bib{AAGHR2015}{article}{
 author={Albiac, F.},
 author={Ansorena, J.~L.},
 author={Garrig{\'o}s, G.},
 author={Hern{\'a}ndez, E.},
 author={Raja, M.},
 title={Conditionality constants of quasi-greedy bases in super-reflexive
 Banach spaces},
 journal={Studia Math.},
 volume={227},
 date={2015},
 number={2},
 pages={133--140},
}

\bib{AAW2017}{article}{
   author={Albiac, F.},
   author={Ansorena, J.~L.},
   author={Wojtaszczyk, P.},
  title={Conditionality constants of quasi-greedy bases in non-superreflexive Banach space},
 journal={Constr. Approx.},
   doi= {10.1007/s00365-017-9399-x},
}


\bib{AlbiacKalton2016}{book}{
 author={Albiac, F.},
 author={Kalton, N.~J.},
 title={Topics in Banach space theory, 2nd revised and updated edition},
 series={Graduate Texts in Mathematics},
 volume={233},
 publisher={Springer International Publishing},
 date={2016},
 pages={xx+508},
 }


\bib{Allen1978}{article}{
 author={Allen, G.~D.},
  title={Duals of Lorentz spaces},
  journal={ Pacific J. Math.},
  volume={177},
 date={1978},
 pages={287--291},
}

\bib{BBGHO2018}{article}{
author={Bern\'a, P.},
 author={Blasco, O.},
 author={Garrig{\'o}s, G.},
 author={Hern{\'a}ndez, E.},
 author={Oikhberg, T.},
 title={The greedy algorithm for non-greedy type bases},
 journal={Submitted},
 }

\bib{BenSha1988}{book}{
 author={Bennett, C.},
 author={Sharpley, R.},
 title={Interpolation of operators},
 series={Pure and Applied Mathematics},
 volume={129},
 publisher={Academic Press, Inc., Boston, MA},
 date={1988},
 pages={xiv+469},
}


\bib{CLL1977}{article}{
   author={Casazza, P.~G.},
   author={Lin, B.~L.},
   author={Lohman, R.~H.},
   title={On James' quasi-reflexive Banach space},
   journal={Proc. Amer. Math. Soc.},
   volume={67},
   date={1977},
   number={2},
   pages={265--271},
   issn={0002-9939},
   review={\MR{0458129}},
}

\bib{DKK2003}{article}{
 author={Dilworth, S.~J.},
 author={Kalton, N.~J.},
 author={Kutzarova, D.},
 title={On the existence of almost greedy bases in Banach spaces},
 note={Dedicated to Professor Aleksander Pe\l czy\'nski on the occasion of
 his 70th birthday},
 journal={Studia Math.},
 volume={159},
 date={2003},
 number={1},
 pages={67--101},
}

\bib{DKKT2003}{article}{
 author={Dilworth, S.~J.},
 author={Kalton, N.~J.},
 author={Kutzarova, D.},
 author={Temlyakov, V.~N.},
 title={The thresholding greedy algorithm, greedy bases, and duality},
 journal={Constr. Approx.},
 volume={19},
 date={2003},
 number={4},
 pages={575--597},
}

\bib{DilworthMitra2001}{article}{
   author={Dilworth, S.~J.},
   author={Mitra, D.},
   title={A conditional quasi-greedy basis of $l_1$},
   journal={Studia Math.},
   volume={144},
   date={2001},
   number={1},
   pages={95--100},
}

\bib{DKW2002}{article}{
   author={Dilworth, S.~J.},
   author={Kutzarova, D.},
   author={Wojtaszczyk, P.},
   title={On approximate $l_1$ systems in Banach spaces},
   journal={J. Approx. Theory},
   volume={114},
   date={2002},
   number={2},
   pages={214--241},
}

\bib{DST2012}{article}{
   author={Dilworth, S.~J.},
   author={Soto-Bajo, M.},
   author={Temlyakov, V.~N.},
   title={Quasi-greedy bases and Lebesgue-type inequalities},
   journal={Studia Math.},
   volume={211},
   date={2012},
   number={1},
   pages={41--69},
}

\bib{GHO2013}{article}{
 author={Garrig{\'o}s, G.},
 author={Hern{\'a}ndez, E.},
 author={Oikhberg, T.},
 title={Lebesgue-type inequalities for quasi-greedy bases},
 journal={Constr. Approx.},
 volume={38},
 date={2013},
 number={3},
 pages={447--470},
 }

\bib{GW2014}{article}{
 author={Garrig{\'o}s, G.},
 author={Wojtaszczyk, P.},
 title={Conditional quasi-greedy bases in Hilbert and Banach spaces},
 journal={Indiana Univ. Math. J.},
 volume={63},
 date={2014},
 number={4},
 pages={1017--1036},
 }
 
 \bib{James1950}{article}{
   author={James, R.~C.},
   title={Bases and reflexivity of Banach spaces},
   journal={Ann. of Math. (2)},
   volume={52},
   date={1950},
   pages={518--527},
}

\bib{JohnsonOdell1974}{article}{
   author={Johnson, W.~B.},
   author={Odell, E.},
   title={Subspaces of $L_{p}$ which embed into $l_{p}$},
   journal={Compositio Math.},
   volume={28},
   date={1974},
   pages={37--49},
}
 
 \bib{JRZ1971}{article}{
   author={Johnson, W.~B.},
   author={Rosenthal, H.~P.},
   author={Zippin, M.},
   title={On bases, finite dimensional decompositions and weaker structures
   in Banach spaces},
   journal={Israel J. Math.},
   volume={9},
   date={1971},
   pages={488--506},
}

\bib{Gogyan2010}{article}{
   author={Gogyan, S.},
   title={An example of an almost greedy basis in $L^1(0,1)$},
   journal={Proc. Amer. Math. Soc.},
   volume={138},
   date={2010},
   number={4},
   pages={1425--1432},
}

\bib{KP1962}{article}{
   author={Kadec, M.~I.},
   author={Pe\l czy\'nski, A.},
   title={Bases, lacunary sequences and complemented subspaces in the spaces
   $L_{p}$},
   journal={Studia Math.},
   volume={21},
   date={1961/1962},
   pages={161--176},
}
\bib{KonyaginTemlyakov1999}{article}{
   author={Konyagin, S. V.},
   author={Temlyakov, V. N.},
   title={A remark on greedy approximation in Banach spaces},
   journal={East J. Approx.},
   volume={5},
   date={1999},
   number={3},
   pages={365--379},
}

\bib{LinPel1968}{article}{
   author={Lindenstrauss, J.},
   author={Pe{\l}czy{\'n}ski, A.},
   title={Absolutely summing operators in $\LL_{p}$-spaces and their
   applications},
   journal={Studia Math.},
   volume={29},
   date={1968},
   pages={275--326},
}

\bib{LinTza1977}{book}{
 author={Lindenstrauss, J.},
 author={Tzafriri, L.},
 title={Classical Banach spaces. I},
 note={Sequence spaces;
 Ergebnisse der Mathematik und ihrer Grenzgebiete, Vol. 92},
 publisher={Springer-Verlag, Berlin-New York},
 date={1977},
 pages={xiii+188},
}


\bib{LinTza1972}{article}{
 author={Lindenstrauss, J.},
 author={Tzafriri, L.},
title={On Orlicz sequence spaces II},
journal={Israel J. Math.},
volume={11},
year={1972},
pages={355--379},
}

\bib{Pel1960}{article}{
 author={Pe{\l}czy{\'n}ski, A.},
 title={Projections in certain Banach spaces},
 journal={Studia Math.},
 volume={19},
 date={1960},
 pages={209--228},
}

\bib{PisierXu1987}{article}{
 author={Pisier, G.},
 author={Xu, Q.~H.},
 title={Random series in the real interpolation spaces between the spaces $v_p$},
 conference={
 title={Geometrical aspects of functional analysis (1985/86)},
 },
 book={
 series={Lecture Notes in Math.},
 volume={1267},
 publisher={Springer, Berlin},
 },
 date={1987},
 pages={185--209},
 }


\bib{Singer1970}{book}{
 author={Singer, I.},
 title={Bases in Banach spaces. I},
 note={Die Grundlehren der mathematischen Wissenschaften, Band 154},
 publisher={Springer-Verlag, New York-Berlin},
 date={1970},
 pages={viii+668},
}

\bib{TemlyakovYangYe2011}{article}{
 author={Temlyakov, V.~N.},
 author={Yang, M.},
 author={Ye, P.},
 title={Lebesgue-type inequalities for greedy approximation with respect
 to quasi-greedy bases},
 journal={East J. Approx.},
 volume={17},
 date={2011},
 number={2},
 pages={203--214},
}

\bib{TemlyakovYangYeB}{article}{
   author={Temlyakov, V.~N.},
   author={Yang, M.},
   author={Ye, P.},
   title={Greedy approximation with regard to non-greedy bases},
   journal={Adv. Comput. Math.},
   volume={34},
   date={2011},
   number={3},
   pages={319--337},
}


\bib{Wo2000}{article}{
 author={Wojtaszczyk, P.},
 title={Greedy algorithm for general biorthogonal systems},
 journal={J. Approx. Theory},
 volume={107},
 date={2000},
 number={2},
 pages={293--314},
}

\end{biblist}
\end{bibsection}
\end{document}